\documentclass[12pt]{article}

\usepackage{amsthm,times,enumerate,amssymb,amsmath,amsfonts}

\usepackage{picinpar}
\usepackage{layout}
\usepackage{verbatim}
\usepackage{epsfig}
\usepackage{setspace}

\newtheorem{theorem}{Theorem}[section]
\newtheorem*{theorem*}{Theorem}
\newtheorem{proposition}[theorem]{Proposition}
\newtheorem{lemma}[theorem]{Lemma}

\newtheorem{corollary}[theorem]{Corollary}
\newtheorem*{corollary*}{Corollary}

\newtheorem*{conjecture*}{Conjecture}

\theoremstyle{remark}

\newtheorem*{remark}{Remark}

\theoremstyle{definition}
\newtheorem{definition}[theorem]{Definition}

\newtheorem{example}[theorem]{Example}
\numberwithin{figure}{section}

\def\R{\mathbb R}
\def\Z{\mathbb Z}

\def\S{\mathcal S}
\def\C{\mathcal C}
\def\Re{\mathcal R}

\title{Global asymptotic stability for a class of nonlinear
  chemical equations} \author{David F. Anderson$^{1}$}

\begin{document}

\maketitle

\footnotetext[1]{Department of Mathematics, University of Wisconsin, Madison,
  WI, 53706}

\begin{abstract}
  We consider a class of nonlinear differential equations that arises
  in the study of chemical reaction systems that are known to be
  locally asymptotically stable and prove that they are in fact
  globally asymptotically stable.  More specifically, we will consider
  chemical reaction systems that are weakly reversible, have a
  deficiency of zero, and are equipped with mass action kinetics.  We
  show that if for each $c \in \R_{> 0}^m$ the intersection of the
  stoichiometric compatibility class $c + S$ with the subsets on the
  boundary that could potentially contain equilibria, $L_W$, are at
  most discrete, then global asymptotic stability follows.  Previous
  global stability results for the systems considered in this paper
  required $(c + S) \cap L_W = \emptyset$ for each $c \in \R^m_{> 0}$,
  and so this paper can be viewed as an extension of those works.
\end{abstract}

\section{Introduction}

This paper is motivated by the consideration of a class of nonlinear
systems that arises in the study of chemistry and biochemistry.
Suppose there are $m$ chemical species, $\{X_1,\dots,X_m\}$,
undergoing a series of chemical reactions.  For a given reaction,
denote by $y, y' \in \Z^m_{\ge 0}$ the vectors representing the number
of molecules of each species consumed and created in one instance of
that reaction, respectively.  Using a slight abuse of notation, we
associate each such $y$ (and $y'$) with a linear combination of the
species in which the coefficient of $X_i$ is $y_i$.  For example, if
$y = [1, \ 2, \ 3]^T$ for a system consisting of three species, we
associate with $y$ the linear combination $X_1 + 2X_2 + 3X_3$.  Under
this association, each $y$ (and $y'$) is termed a {\em complex} of the
system.  We may now denote any reaction by the notation $y \to y'$,
where $y$ is the source, or reactant, complex and $y'$ is the product
complex.  We note that each complex will typically appear as both a
source complex and a product complex in the system. Let $\S = \{X_i\},
\C = \{y\},$ and $\Re = \{y \to y'\}$ denote the sets of species,
complexes, and reactions, respectively.  Denote the concentration
vector of the species as $x \in \R^m$.  In order to know how the state
of the system is changing, we need to know the rate at which each
reaction is taking place.  Therefore, for each reaction, $y \to y'$,
there is a $C^1$ function $R_{y \to y'}(\cdot)$ satisfying the
following:
\begin{enumerate}
\item $R_{y \to y'}(\cdot)$ is a function of the concentrations of
  those species contained in the source complex, $y$.
\item $R_{y \to y'}(\cdot)$ is monotone increasing in each of its
  inputs and $R_{y \to y'}(x) = 0$ if {\em any} of its inputs are
  zero.
\end{enumerate}
The dynamics of the system are then given by
\begin{equation}
  \dot x(t) = \sum_{y \to y' \in \Re} R_{y \to y'}(x(t))(y' - y) \
  \dot{=} \ f(x(t)),
  \label{eq:main}
\end{equation}
where the last equality is a definition.  The functions $R_{y \to y'}$
are typically referred to as the {\em kinetics} of the system.  This
notation closely matches that of Feinberg, Horn, and Jackson, and it
is their results that the main theorem in this paper extends
(\cite{HornJack72, Feinberg72, Feinberg87, Feinberg95a,
  FeinbergLec79}).

Integrating equation \eqref{eq:main} gives
\begin{equation*}
  x(t) = x(0) + \sum_{y \to y' \in \Re} \left(\int_0^t
    R_{y \to y'}(x(s)) ds \right) (y' - y).
\end{equation*}
Therefore, $x(t) - x(0)$ remains in the linear space $S = $ Span$\{y'
- y\}_{y \to y' \in \Re}$ for all time.  We shall refer to the space
$S$ as the stoichiometric subspace of the system and refer to the sets
$c + S$, for $c \in \R^m$, as stoichiometric compatibility classes, or
just compatibility classes.  Later we will demonstrate that
trajectories with positive initial conditions remain in $\R^m_{>0}$
for all time.  The sets $(c + S) \cap \R^m_{>0}$ will therefore be
referred to as the {\em positive stoichiometric compatibility
  classes}.  Given that trajectories remain in their positive
stoichiometric compatibility classes for all time, we see that the
types of questions that one should ask about these systems differ from
the questions one normally asks about nonlinear systems.  For example,
instead of asking whether there is a unique equilibrium value to the
system \eqref{eq:main}, and then asking about its stability
properties, it is clearly more appropriate to ask whether there is a
unique equilibrium {\em within each positive stoichiometric
  compatibility class} and, if so, what are its stability properties
{\em relative to its compatibility class}.

The most common kinetics chosen is that of {\em mass action kinetics}.
A chemical reaction system is said to have mass action kinetics if
\begin{equation}
  R_{y \to y'}(x) = \ k_{y \to y'}x_1^{y_1} x_2^{y_2} \cdots x_m^{y_m}, 
  \label{eq:massaction}
\end{equation}
for some constant $k_{y \to y'}$.  It has been shown that for many
systems of the form \eqref{eq:main} with mass action kinetics, there
is within each positive stoichiometric compatibility class precisely
one equilibrium and that equilibrium is locally asymptotically stable
relative to its class (\cite{HornJack72, Feinberg87, FeinbergLec79}).
In order to show that the equilibrium values are locally stable, the
following Lyapunov function is used (one for each compatibility class)
\begin{equation}
  V(x,\bar x) = V(x) = \sum_{i=1}^m \left[ x_i(\ln(x_i) - \ln(\bar
    x_i) - 1) + \bar x_i \right],
  \label{eq:lyapunov}
\end{equation}
where $\bar x$ is the unique equilibrium of a given positive
stoichiometric compatibility class.  It turns out that the function
$V$ ``almost'' acts as a global Lyapunov function.  That is, $V$ is
non-negative for $x \in (\bar x + S) \cap \R^m_{>0}$, zero only at
$\bar x$, and strictly decreases along trajectories.  However, $V$
does not tend to infinity as trajectories near the boundary of $(\bar
x + S) \cap \R^m_{>0}$, and without such unboundedness one can not, in
general, conclude global stability.  It has been shown in numerous
papers, however, that global stability of $\bar x$ does hold if there
are no equilibria on the boundary of $(\bar x + S) \ \cap \R^m_{>0}$
(\cite{Sontag2001, Sontag2007, Siegel1994, Siegel2000}).  Therefore,
work has been done giving sufficient conditions for the non-existence
of boundary equilibria in order to conclude that the equilibrium value
within each compatibility class is globally stable relative to its
class (\cite{Sontag2007, Siegel1994, Siegel2000}).

To each subset $W$ of the set of species, the set of boundary points
$L_W$ is defined to be
\begin{equation}
  L_W = \{x \in \R^{m} : x_i = 0 \Leftrightarrow X_i \in W\}.
  \label{eq:boundary}
\end{equation}
We will show that there are no boundary equilibria if and only if
\begin{equation}
  [(c + S) \cap \R^m_{\ge 0}] \cap L_W = \emptyset,
\label{eq:emptyintersection}
\end{equation}
for all $c \in \R^m_{>0}$ and for certain subsets of the species, $W$.
We will then prove that global stability holds if the intersection
given in equation \eqref{eq:emptyintersection} is either empty or
discrete for each $c \in \R^m_{>0}$ and those same subsets, $W$.  This
will imply that global stability holds even if there are boundary
equilibria, so long as the boundary equilibria are extreme points of
the positive stoichiometric compatibility classes.  To the best of our
knowledge there is only one other result concerning the global
stability of mass action systems with boundary equilibria, and is
contained within the Ph.D.  thesis of Madalena Chavez
(\cite{ChavezThesis}).  In order to guarantee global stability even if
there exist boundary equilibria, she requires that each boundary
equilibrium be hyperbolic with respect to its stoichiometric
compatibility class and she requires another (more technical)
condition on the stable subspaces of each boundary equilibrium (see
\cite{ChavezThesis}, pg.  106 for details).  As our result are
applicable to systems with boundary equilibria that are
non-hyperbolic, our results can be viewed as an extension of those in
\cite{ChavezThesis}.

The layout of the paper is as follows.  In Section \ref{sec:model} we
will introduce the systems we consider in this paper: weakly
reversible, deficiency zero systems with mass action kinetics.  We
will then present some preliminary results and will conclude with a
proof that global stability follows if there are no equilibria on the
boundary of the positive stoichiometric compatibility classes. No
originality is claimed for this result as it is known. Also in Section
\ref{sec:model} we demonstrate how the ``no boundary equilibria''
assumption is equivalent to equation \eqref{eq:emptyintersection}
holding for all $c \in \R^m_{>0}$ and certain subsets of the species,
$W$.  In Section \ref{sec:results} we extend the previous theorems to
prove that global stability still holds if the intersection given in
equation \eqref{eq:emptyintersection} is always either empty {\em or}
discrete for those same subsets, $W$.  We also show in Section
\ref{sec:results} how the hypothesis that the intersection in equation
\eqref{eq:emptyintersection} is always empty or discrete is equivalent
to a condition on the extreme points of the non-negative
stoichiometric compatibility classes.  In Section \ref{sec:examples},
we demonstrate our results on a number of examples.  Finally, in
Section \ref{sec:more} we sketch how to extend our results to systems
with non-mass action kinetics.

\section{Preliminary results}
\label{sec:model}

We start with definitions taken from \cite{HornJack72},
\cite{FeinbergLec79}, and \cite{Feinberg87}.

\begin{definition}
  A chemical reaction network, $\{\S,\C,\Re \}$, is called {\em weakly
    reversible} if for any reaction $y \to y'$, there is a sequence of
  directed reactions beginning with $y'$ and ending with $y$.  That
  is, there exists complexes $y_1,\dots,y_k$ such that the following
  reactions are in $\Re$: $y' \to y_1$, $y_1 \to y_2$, $\dots$, $y_k
  \to y$.
  \label{def:WR}
\end{definition}

To each reaction system, $\{\S,\C,\Re\}$, there is a unique, directed
graph constructed in the following manner.  The nodes of the graph are
the complexes, $\C$.  A directed edge is then placed from complex $y$
to complex $y'$ if and only if $y \to y' \in \Re$.  Each connected
component of the resulting graph is termed a {\em linkage class} of
the graph.  We denote the number of linkage classes by $l$.

As shown in the introduction, each trajectory remains in its
stoichiometric compatibility class for all time.  There is another
restriction on the trajectories of solutions to equation
\eqref{eq:main} that is given in the following lemma.  The proof can
be found in both \cite{Sontag2001} and \cite{AndThesis}.
\begin{lemma}
  Let $x(t)$ be a solution to \eqref{eq:main} with initial condition
  $x(0) \in \R^m_{> 0}$.  Then, $x(t) \in \R^m_{> 0}$, for all $t > 0$.
  \label{lemma:pos}
\end{lemma}

\subsection{Persistence and $\omega$-limit points}

By Lemma \ref{lemma:pos}, each trajectory must remain within
$\R^m_{>0}$ if its initial condition is in $\R^m_{>0}$, therefore the
linear subsets of interest are the intersections of the stoichiometric
compatibility classes and $\R^m_{>0}$.  Recall that in the
Introduction these sets were termed the positive stoichiometric
compatibility classes.  This paper will mainly be concerned with
showing that trajectories to systems given by equation \eqref{eq:main}
remain away from the boundary of the positive stoichiometric
compatibility classes.  That is, we will show that the systems are
{\em persistent}.  To be precise, let $\phi(t,\xi)$ be the solution to
equation \eqref{eq:main} with initial condition $\xi \in \R^m_{>0}$.
The set of $\omega$-limit points of the trajectory is
\begin{equation}
  \omega(\xi) \ \dot = \ \{x \in \R^m_{\ge 0} : \phi(t_n,\xi) \to x,
  \hbox{ for some } t_n \to \infty \}.     
  \label{omega}
\end{equation}
\begin{definition} A system is {\em persistent} if $\omega(\xi) \cap
  \partial \R^{m}_{>0} = \emptyset$, for each $\xi \in \R^{m}_{>0}$.
  \label{def:persistence}
\end{definition}
We refer the reader to \cite{Butler86a, Butler86b, Card80, Thieme2000}
for some of the history and usage of the notion of persistence in the
study of dynamical systems.  In order to show that a chemical system
is persistent, it is critical to understand which points on the
boundary are capable of being $\omega$-limit points.  With that in
mind, we introduce the following definition.

\begin{definition}
  A nonempty subset $W$ of the set of species is called a {\em
    semi-locking set} if for each reaction in which there is an
  element of $W$ in the product complex, there is an element of $W$ in
  the reactant complex.  $W$ is called a {\em locking set} if every
  reactant complex contains an element of $W$.
  \label{def:locking}
\end{definition}

Locking and semi-locking sets are easily understood.  Suppose that $W
\subset \{X_1,\dots,X_m\}$ is a locking set.  Then, because every
reactant complex contains an element of $W$, if the concentration of
each element of $W$ is zero, each kinetic function, $R_{y \to y'}$,
must equal zero.  Therefore, all of the fluxes are zero, and $\dot
x(t) = 0$.  We therefore see that the system is ``locked'' in place.
Now suppose $W$ is a semi-locking set.  If the concentration of each
element of $W$ is zero, then any flux which affects the species of $W$
is turned off and the elements of $W$ are ``locked'' at zero.
Semi-locking sets have another, important, interpretation in terms of
the linkage classes and weak reversibility.  If the concentrations of
the elements of a semi-locking set are equal to zero and the system is
weakly reversible, then all of the fluxes of any linkage class with a
complex containing an element of $W$ are equal to zero (and so these
linkage classes are ``locked''), while the fluxes of the other linkage
classes are all not equal to zero.  Therefore, the concept of a
semi-locking set and a locking set are equivalent for systems that are
weakly reversible and have only one linkage class.  We note that our
notion of a semi-locking set is analogous to the concept of a {\em
  siphon} in the theory of Petri nets.  See \cite{Sontag2007} for a
full discussion, including historical references, of the role of Petri
nets in the study of chemical reaction networks.

The following Theorem now characterizes the boundary points that have
the capability of being $\omega$-limit points of the system.  This
result was first proved in \cite{Sontag2007}, however the proof given
here is completely different and straightforward.

\begin{theorem}
  Let $W$ be a non-empty subset of the species. If there exists a $\xi
  \in \R^m_{>0}$ such that $\omega(\xi)\cap L_{W} \ne \emptyset$, then
  $W$ is a semi-locking set.
  \label{thm:omegapoints}
\end{theorem}

\begin{proof}
  Suppose, in order to find a contradiction, that there is a $\xi \in
  \R^{m}_{>0}$ and a subset of the species, $W$, such that
  $\omega(\xi) \cap L_W \ne \emptyset$ and $W$ is not a semi-locking
  set. Let $y \in \omega(\xi) \cap L_W$.  We note that there exists a
  species $X_j$, with $X_j \in W$, such that at least one input to
  $X_j$ (term in $f_j$ of equation \eqref{eq:main} with a positive
  coefficient) is non-zero if the concentrations are given by $y$, for
  otherwise $W$ would be a semi-locking set.  Therefore, because all
  outputs from species $X_j$ (terms in $f_j$ with a negative
  coefficient) are zero at $y$, there exists $\epsilon > 0$ and $k >
  0$ such that if $x(t) \in \R^{m}_{>0} \cap B_{\epsilon}(y)$, then
  \begin{equation}
    f_j(x(t)) = x_j'(t) > k,
    \label{jderiv}
  \end{equation}
  where $B_{\epsilon}(y) = \{x : |x - y| < \epsilon\}$. 

  Because $f(\cdot)$ is $C^1$, we have $\|f\|_{\infty,loc} < M$ for
  some $M>0$, and this bound is valid in $\R^{m}_{>0} \cap
  B_{\epsilon}(y)$.  Therefore, for any $0 < a < b$, if $x(t) \in
  \R^{m}_{>0} \cap B_{\epsilon}(y)$ for $t \in (a,b)$, we have that
  \begin{equation}
    |x(b) - x(a)| = \left| \int_a^b f(x(s))ds \right| \le (b-a)M.
    \label{epsbound}
  \end{equation}
  Now consider a partial trajectory starting on the boundary of
  $\R^{m}_{>0} \cap B_{\epsilon}(y)$ at time $t_{\epsilon}$, ending on
  the boundary of $\R^{m}_{>0} \cap B_{\epsilon/2}(y)$ at time
  $t_{\epsilon/2}$, and remaining within that annulus for all time in
  $(t_{\epsilon},t_{\epsilon/2})$.  Note that one such partial
  trajectory must exist every time we enter $\R^{m}_{>0} \cap
  B_{\epsilon/2}(y)$, and this happens at least once by our assumption
  that $y \in \omega(\xi) \cap L_W$.  By \eqref{epsbound},
  $t_{\epsilon/2} - t_{\epsilon} \ge \epsilon/(2M)$.  On the other
  hand, by \eqref{jderiv}, $x_j'(t) > k$ for $t \in
  (t_{\epsilon},t_{\epsilon/2})$.  Therefore,
  \begin{align*}
    x_j(t_{\epsilon/2}) &= x_j(t_{\epsilon}) +
    \int_{t_{\epsilon}}^{t_{\epsilon/2}}x_j'(s)ds\\
    &\ge x_j(t_{\epsilon}) + \epsilon k/(2M)\\
    &\ge \epsilon k /(2M).
  \end{align*}
  Combining the above with the fact that we still have $x_j'(t) > k$
  on $\R^{m}_{>0} \cap B_{\epsilon/2}(y)$, and we see that there can
  not exist times $t_n$ such that $x(t_n) \to y$, as $n \to \infty$.
  This is a contradiction and completes the proof.
\end{proof}

\begin{remark}
  Theorem \ref{thm:omegapoints} is a powerful tool for understanding
  the dynamics of chemical reaction systems.  We see that in order to
  prove that a chemical system is persistent, it is sufficient to show
  that $[(c + S) \cap \R^m_{\ge 0}] \cap L_W = \emptyset$ for all $c
  \in \R^m_{>0}$ and all semi-locking sets $W$.  We will show in Lemma
  \ref{thm:noequi} that for many reaction systems such a condition is
  equivalent to having no equilibria on the boundaries of the positive
  stoichiometric compatibility classes.
\end{remark}

\subsection{Deficiency and the Deficiency Zero Theorem}

We require one more definition before we can state precisely the types
of systems we consider in this paper.

\begin{definition}
  The {\em deficiency}, $\delta$, of a reaction network $\{\S,\C,\Re
  \}$ is given by $\delta = n - l - s$, where $n$ is the number of
  complexes of the system, $l$ is the number of linkage classes, and
  $s = $dim $S$, the dimension of the stoichiometric subspace.
\end{definition}

\begin{remark}
  It has been shown that the deficiency of a reaction network is a
  non-negative number.  In fact, the deficiency is the dimension of a
  certain subspace associated with the system.  See \cite{Feinberg87},
  \cite{Feinberg95a}, and \cite{FeinbergLec79} for details.
\end{remark}

The main types of systems considered in this paper are those with mass
action kinetics, are weakly reversible, and have a deficiency of zero.
The following theorem by Feinberg (\cite{FeinbergLec79, Feinberg95a})
is the catalyst for studying such systems.  The proof can be found in
\cite{FeinbergLec79} or \cite{Feinberg95a}.

\begin{theorem}[The Deficiency Zero Theorem]
  Consider a system of the form \eqref{eq:main} with mass action
  kinetics that is weakly reversible and has a deficiency of zero.
  Then, within each positive stoichiometric compatibility class there
  is precisely one equilibrium value and it is locally asymptotically
  stable relative to its compatibility class.
  \label{thm:def0}
\end{theorem}

In order to prove that the systems considered in the Deficiency Zero
Theorem are locally asymptotically stable, the Lyapunov function
\eqref{eq:lyapunov} is used. It is shown that for $x \in (\bar x + S)
\cap \R^m_{>0}$ (where $\bar x$ is the equilibrium guaranteed to exist
by Theorem \ref{thm:def0}), $V(x) \ge 0$ with equality if and only if
$x = \bar x$, and that $dV(x(t))/dt < 0$ for all trajectories with
initial condition in $(\bar x + S) \cap \R^m_{>0}$.  We will make use
of these facts throughout the paper without reference, however we
point the interested reader to the original works for details.

\subsection{Boundary equilibria}

The following Lemma shows that having no equilibria on the boundaries
of the positive stoichiometric compatibility classes is equivalent to
$[(c + S) \cap \R^m_{\ge 0}] \cap L_W = \emptyset$ for all $c \in
\R^m_{>0}$ and semi-locking sets $W$.  Following Lemma
\ref{thm:noequi} we present a theorem pertaining to any system
equipped with a globally defined Lyapunov function that does not
necessarily go to infinity as $x$ goes to the boundary of the domain.
We then use these results in combination with Theorem
\ref{thm:omegapoints} to conclude that for weakly reversible,
deficiency zero systems with mass action kinetics, having no
equilibria on the boundary of the positive stoichiometric
compatibility classes implies global stability of the equilibrium
values given by Theorem \ref{thm:def0}.  We again note that it is
already known that global asymptotic stability follows from a lack of
boundary equilibria.  For example, in \cite{Sontag2001} Sontag showed
that all trajectories must converge to the set of equilibria, and so a
lack of boundary equilibria implies convergence to the unique
equilibrium in the interior of the positive stoichiometric
compatibility class.  We rederive this result here because our methods
put it in a larger context in which global stability is understood
through the intersections given in equation
\eqref{eq:emptyintersection} and because it makes clear how our
results in Section \ref{sec:results} are truly a generalization of
this fact.

\begin{lemma}
  For any chemical reaction system, the set of equilibria on the
  boundaries of the positive stoichiometric compatibility classes is
  contained in $\bigcup_{c}\bigcup_W [(c + S) \cap \R^m_{\ge 0}] \cap
  L_W$, where the first union is over $c \in \R^m_{>0}$ and the second
  union is over the semi-locking sets.  Further, if there are no
  equilibria on the boundaries of the positive stoichiometric
  compatibility classes for a weakly reversible, deficiency zero
  system with mass action kinetics, then $[(c + S) \cap \R^m_{\ge 0}]
  \cap L_W = \emptyset$ for all $c \in \R^m_{>0}$ and semi-locking
  sets $W$.
  \label{thm:noequi}
\end{lemma}

\begin{proof}
  Let $y$ be an equilibrium on the boundary of a positive
  stoichiometric compatibility class.  Let $W$ be the set of species
  with a concentration of zero at $y$.  Because each complex that
  contains an element of $W$ is providing zero flux, in order for $y$
  to be an equilibrium value each reaction in which there is an
  element of $W$ in the product complex must have an element of $W$ in
  the reactant complex.  Thus, $W$ is a semi-locking set and $y \in
  [(c + S) \cap \R^m_{\ge 0}] \cap L_W$, for some $c \in \R^m_{> 0}$.

  In order to prove the second part of the Lemma, we suppose $W$ is a
  semi-locking set for the system and suppose $y \in [(c + S) \cap
  \R^m_{\ge 0}] \cap L_W$ for some $c \in \R^m_{>0}$. We will now
  produce an equilibrium value on the boundary of $(c + S) \cap
  \R^m_{>0}$.  If $W = \mathcal{S}$, $y = \vec 0$, and, because $W$ is
  a semi-locking set, $y$ is an equilibrium.  Otherwise, consider the
  system consisting only of those species not in the semi-locking set
  $W$.  By the arguments in \cite{Feinberg95a}, the linkage classes
  not ``locked'' by $W$ form their own weakly reversible, deficiency
  zero system.  Therefore, there is an equilibrium for that reduced
  system, $\bar z$.  Let $\bar y = (\bar z, \vec{0})$ (where we have
  potentially rearranged the ordering of the species so that those not
  in the semi-locking set came first).  $\bar y$ is a boundary
  equilibrium value to our original system.  Therefore, the result is
  shown.
\end{proof}

\begin{theorem}
  Let $x(t) = x(t,x(0))$ be the solution to $\dot x = f(x)$ with
  initial condition $x(0)$, where $f$ is $C^1$ and the domain of
  definition of the system is the open set $C \subset \R^m$.  Let
  $\bar x \in C$ be the unique equilibrium value to the system.
  Finally, suppose that there is a globally defined Lyapunov function
  $V$ that satisfies:
  \begin{enumerate}
  \item $V(x) \ge 0$ with equality if and only if $x = \bar x$.
  \item $dV(x(t))/dt \le 0$ with equality if and only if $x(t) = \bar
    x$.
    \item $V(x) \to \infty$, as $|x| \to \infty$.
  \end{enumerate}
  Then either $x(t) \to \bar x$ or $x(t) \to \partial C$.
  \label{thm:converge}
\end{theorem}

\begin{proof}
  Suppose that $x(t) \nrightarrow \bar x$.  Because $V(\cdot)$
  decreases along trajectories, the value $V(x(t))$ is bounded above
  by $V(x(0))$ for all $t > 0$.  Therefore, because $V(x) \to \infty$
  as $|x| \to \infty$, $x(t)$ remains bounded for all $t > 0$. Also,
  the local asymptotic stability of $\bar x$ combined with the fact
  that $x(t) \nrightarrow \bar x$ implies there is a $\rho >0$ such
  that $|x(t) - \bar x|>\rho$ for all $t >0$.

  Let $\epsilon > 0$ and for $x \in C$ let $d(x,\partial C)$ represent
  the distance from $x$ to the boundary of $C$.  Let $C_{\epsilon} =
  \{x \in C \ | \ d(x,\partial C) \ge \epsilon \hbox{ and } |x - \bar
  x| \ge \rho\}$.  Using that trajectories remain bounded for all time,
  we may use the continuity of the functions $\nabla V$ and $f$ to
  conclude that there is a positive number $\eta = \eta(\epsilon)$
  such that $\nabla V(x) \cdot f(x) < -\eta$ for all $x \in
  C_{\epsilon}$.  Therefore, the maximum amount of time that any
  trajectory can spend in the set $C_{\epsilon}$ is $V(x(0))/\eta$
  (for, otherwise, $x(t) \to \bar x$).  Because $\epsilon > 0$ was
  arbitrary we see that $x(t) \to \partial C$.
\end{proof}

\begin{corollary}
  If there are no equilibria on the boundaries of the positive
  stoichiometric compatibility classes for a weakly reversible
  deficiency zero system with mass action kinetics, then the unique
  positive equilibrium value within each positive stoichiometric
  compatibility class is globally asymptotically stable relative to
  its compatibility class.
  \label{cor:noequiglobal}
\end{corollary}

\begin{proof}
  This is a direct result of Theorems \ref{thm:omegapoints} and
  \ref{thm:converge} and Lemma \ref{thm:noequi}.
\end{proof}

\section{Main results}
\label{sec:results}

By Lemma \ref{thm:noequi}, we see that the no boundary equilibria
assumption for weakly reversible deficiency zero systems with mass
action kinetics is equivalent to the assumption that $[(c + S) \cap
\R^m_{\ge 0}] \cap L_W = \emptyset$, for all $c \in \R^m_{>0}$ and all
semi-locking sets $W$.  This then implies global stability by
Corollary \ref{cor:noequiglobal}.  We will extend these results by
proving that global stability holds if $[(c + S) \cap \R^m_{\ge 0}]
\cap L_W$ is empty {\em or} discrete for each $c \in \R^m_{>0}$ and
each semi-locking set $W$.  The following definition is necessary.

\begin{definition}
  For a vector $x \in \R^m$, the {\em support} of $x$, denoted
  supp$(x)$, is the subset of the species such that $X_i \in
  \hbox{supp}(x)$ if and only if $x_i \ne 0$.
\end{definition}

\begin{proposition}
  Let $\{\S,\C,\Re\}$ be a weakly reversible, deficiency zero, mass
  action chemical reaction system with dynamics given by equation
  \eqref{eq:main}.  Suppose that $y \in \omega(x(0))$ for some $x(0)
  \in \R^m_{>0}$.  Then there must exist a nonzero $z_0 \in S$ with
  supp$(z_0) \subset$ supp$(y)$.
  \label{prop:main}
\end{proposition}

\begin{proof}
  If $y \in \R^m_{>0}$, there is nothing to show.  Therefore, assume
  that $y$ is on the boundary of the positive stoichiometric
  compatibility class.  By Theorem \ref{thm:omegapoints}, there is a
  semi-locking set $W$ such that $y \in [(x(0) + S) \cap \R^m_{\ge 0}]
  \cap L_W$.

  Let $V(x) : \R^m_{>0} \to \R$, be given by equation
  \eqref{eq:lyapunov} and let
  \begin{equation*}
    V_i(x_i) = x_i(\ln(x_i) - \ln(\bar x_i) - 1) + \bar x_i,
  \end{equation*}
  so that $V(x) = \sum_{i=1}^m V_i(x_i)$.  Reordering the species if
  necessary, we suppose $W = \{X_1,\dots,X_d\}$.  Choose $\rho > 0$ so
  small that for each $i \le d$, $x_i < \rho \implies \ln(x_i) -
  \ln(\bar x_i) < 0$.  Let $\epsilon > 0$ satisfy $\epsilon < \rho$.
  Let $t_{\epsilon}$ be a time such that $x_i(t_{\epsilon}) \le
  \epsilon$ for all $i \le d$ and $|x_j(t_{\epsilon}) - y_j| <
  \epsilon$ for all $j \ge d+1$.  Let $T_{\epsilon} = \min\{t >
  t_{\epsilon} : |x_i(t) - y_i| \le x_i(t_{\epsilon})/2, \hbox{ for
    all } i \le m\}$.  We know such $t_{\epsilon}$ and $T_{\epsilon}$
  exist because $y$ is an $\omega$-limit point of the system.  Note
  that $T_{\epsilon} > t_{\epsilon}$ and that for each $i \le d$,
  $x_i(T_{\epsilon}) < x_i(t_{\epsilon})$.  We consider how $V(x(t))$
  changes from time $t_{\epsilon}$ to time $T_{\epsilon}$.  Applying
  the mean value theorem to each $V_i(\cdot)$ term gives
  \begin{align}
    V(x(T_{\epsilon})) - V(x(t_{\epsilon})) &= \sum_{i=1}^m
    V_i(x_i(T_{\epsilon})) - V_i(x_i(t_{\epsilon}))\\
    \begin{split}    
      &= \sum_{i=1}^d (\ln(\tilde x_i) - \ln(\bar x_i))
      (x_i(T_{\epsilon}) -
      x_i(t_{\epsilon}))\\
      & \ \ \ + \sum_{i=d+1}^m (\ln(\tilde x_i) - \ln(\bar x_i))
      (x_i(T_{\epsilon}) - x_i(t_{\epsilon}))
    \end{split}
    \label{eq:double}
  \end{align}
  for some $\tilde x_i \in [x_i(T_{\epsilon}),x_i(t_{\epsilon})]$.
  Recalling that $V$ decreases along trajectories of $x(t)$ by Theorem
  \ref{thm:def0}, we have $V(x(T_{\epsilon})) - V(x(t_{\epsilon})) <
  0$.  Note that because for $j \ge d+1$, we have $|\tilde x_j - y_j|
  < \epsilon$, there are positive constants $c_j$ such that $c_j >
  |\ln(\tilde x_j) - \ln \bar x_j|$, and that bound is valid {\em for
    any} $\epsilon < \rho$. Let $C = \sum_{j=d+1}^m c_j$.

  By our choices above, we know that for each $i \in \{1,\dots,d\}$
  the following inequalities hold:
  \begin{enumerate}
  \item $\ln(\tilde x_i) - \ln (\bar x_i) < 0$.
  \item  $x_i(T_{\epsilon}) - x_i(t_{\epsilon}) <0$.
  \end{enumerate}
  Therefore, each piece of the first sum in equation \eqref{eq:double}
  is strictly positive.  Thus, to ensure that $V$ is decreasing along
  this trajectory, the second sum in equation \eqref{eq:double} must
  be negative and, letting $\Delta x_i = x_i(T_{\epsilon}) -
  x_i(t_{\epsilon})$ for each $i$, we have
  \begin{align}
    \begin{split}
      \sum_{i=1}^d \left( \ln(\tilde x_i) - \ln (\bar x_i) \right) \Delta x_i 
      &< \left| \sum_{j=d+1}^m \left( \ln(\tilde x_j) - \ln (\bar x_j) \right) 
        \Delta x_j \right| \\
      &\le \sum_{j=d+1}^m c_j |\Delta x_j|.
    \end{split}
    \label{eq:firstbound}
  \end{align}
  In fact, because each term on the left hand side of equation
  \eqref{eq:firstbound} is positive, a similar inequality must hold
  for each $i = 1,\dots,d$.  That is, for $i \le d$
  \begin{align*}
    (\ln(\tilde x_i) - \ln (\bar x_i)) \Delta x_i &\le \sum_{j=d+1}^m
    c_j |\Delta x_j|.
  \end{align*}
  For each $i \le d$, $\tilde x_i \in
  [x_i(T_{\epsilon}),x_i(t_{\epsilon})]$ and
  $x_i(T_{\epsilon}),x_i(t_{\epsilon}) < \epsilon$. Hence, letting
  $|\ln \bar x_i| = k_i$ we have that for each $i \le d$
  \begin{equation*}
    | \ln(\tilde x_i)  - \ln (\bar x_i) | \ge |\ln \epsilon| - k_i.
  \end{equation*}
  Thus, for each $i = 1,\dots,d$,
  \begin{equation*}
    |\Delta x_i| \le \frac{1}{|\ln \epsilon| - k_i} \sum_{j=d+1}^m c_j
    |\Delta x_j|. 
  \end{equation*}
  Let $\Delta_{max} = \sup_{j \in \{d+1,\dots,m\}} \{|\Delta x_j|\}$
  and $\delta(\epsilon) = \sup_{i \in \{1,\dots,d\}} (|\ln \epsilon| -
  k_i)^{-1}$.  We know $\Delta_{max} \ne 0$ because if it were equal
  to zero, then the right hand side of equation \eqref{eq:firstbound}
  would be zero, which it can not be as it is strictly larger than the
  left hand side.  Combining the above shows that for each $i =
  1,\dots,d$,
  \begin{equation*}
    |\Delta x_i| \le \delta(\epsilon) C \Delta_{max}.
  \end{equation*}
  Now we consider the vector $\Delta x = x(T_{\epsilon}) -
  x(t_{\epsilon}) \in S$.  Normalizing the vector $\Delta x$ by
  dividing each entry by $\Delta_{max}$ then produces a vector
  $v(\epsilon) \ \dot = \ \frac{1}{\Delta_{max}} \Delta x$ with the
  following properties:
  \begin{enumerate}
  \item  $v(\epsilon) \in S$.
  \item For $i=1,\dots,d$, $|v_i(\epsilon)| \le \delta(\epsilon)C$.
  \item There is at least one entry in $v(\epsilon)$ with norm $1$
    (the one for which the maximum in the definition of $\Delta_{max}$
    was achieved), and none have a higher norm.
  \item $1 \le |v(\epsilon)| \le m$.
  \end{enumerate}
  Property 4 follows from property 3.  $\epsilon >0$ was arbitrary, so
  we may consider a sequence $\{\epsilon_n\}$ such that $\epsilon_n >
  \epsilon_{n+1}$ and $\epsilon_n \to 0$.  For each $\epsilon_n$ we
  may redo the work above.  This leads to a sequence of vectors
  $\{v(\epsilon_n)\}$ and a sequence of numbers
  $\{\delta(\epsilon_n)\}$ such that $\delta(\epsilon_n) \to 0$ and
  for each $n$ all four properties above hold.  Because each vector
  from the sequence $\{v(\epsilon_n)\}$ is contained in the compact
  space $\{v : 1 \le |v| \le m \} \cap S$, there is a convergent
  subsequence $\{ v(\epsilon_{n_k}) \}$ and a vector $z_0$ such that
  $v(\epsilon_{n_k}) \to z_0 \in \{v : 1 \le |v| \le m \} \cap S
  \subset S$, as $k \to \infty$.  Note that $z_0$ can not be the zero
  vector because $|z_0| > 1$.  However, $\delta(\epsilon_{n_k}) \to
  0$, and so the first $d$ components of $z_0$ are equal to zero.
  Hence, supp$(z_0) \subset$ supp$(y)$.
\end{proof}

\begin{theorem}
  Let $\{\S,\C,\Re\}$ be a weakly reversible, deficiency zero, mass
  action chemical reaction system with dynamics given by equation
  \eqref{eq:main}.  Suppose that for each $c \in \R^m_{>0}$ and each
  semi-locking set $W$, the set $[(c + S) \cap \R^m_{\ge 0}] \cap L_W$
  is either empty or discrete.  Then the unique positive equilibrium
  of each stoichiometric compatibility class guaranteed to exist by
  the Deficiency Zero Theorem is globally asymptotically stable
  relative to its compatibility class.
  \label{thm:main}
\end{theorem}

\begin{proof}
  
  We suppose, in order to find a contradiction, that there is a
  positive equilibrium, $\bar x$, that is not globally asymptotically
  stable relative to its compatibility class.  By Theorems
  \ref{thm:omegapoints}, \ref{thm:def0}, and \ref{thm:converge}, there
  is a semi-locking set $W$, a $\xi \in \R^m_{>0}$, and a vector $y$
  such that $y \in [(\bar x + S) \cap \R^m_{\ge 0}] \cap L_W$ and $y
  \in \omega(\xi)$.  By Proposition \ref{prop:main}, there exists a
  nonzero $z_0 \in S$ such that supp$(z_0) \subset$ supp$(y)$.
  Because $y \in \bar x + S$ and $z_0 \in S$, for any $\eta > 0$ we
  have $y + \eta z_0 \in \bar x + S$.  Further, because supp$(z_0)
  \subset$ supp$(y)$, if $\eta$ is small enough we have that $y + \eta
  z_0 \in \R^m_{\ge 0} \cap L_W$.  But this is valid for all $\eta$
  small enough, and so $[(\bar x + S) \cap \R^m_{\ge 0}] \cap L_W$ is
  not discrete.  This is a contradiction and so the result is shown.
\end{proof}

\begin{corollary}
  Suppose that for a weakly reversible deficiency zero chemical
  reaction system with mass action kinetics, each semi-locking set is
  a locking set.  Suppose further that within each stoichiometric
  compatibility class, the set of equilibria on the boundary is
  discrete.  Then the unique positive equilibrium of each
  stoichiometric compatibility class guaranteed to exist by the
  Deficiency Zero Theorem is globally asymptotically stable relative
  to its compatibility class.
  \label{cor:locking}
\end{corollary}

\begin{proof}
  Because each semi-locking set is a locking set, the set of boundary
  equilibria for a given compatibility class is precisely given by
  $\bigcup_W [(c + S) \cap \R^m_{\ge 0}] \cap L_W$, where the union is
  over the set of semi-locking sets.  Therefore, each $[(c + S) \cap
  \R^m_{\ge 0}] \cap L_W$ is discrete and invoking Theorem
  \ref{thm:main} completes the proof.
\end{proof}

\begin{corollary}
  Suppose that a weakly reversible deficiency zero chemical reaction
  system with mass action kinetics has only one linkage class.
  Suppose further that within each stoichiometric compatibility class,
  the set of equilibria on the boundary is discrete.  Then the unique
  positive equilibrium of each stoichiometric compatibility class
  guaranteed to exist by the Deficiency Zero Theorem is globally
  asymptotically stable relative to its compatibility class.
  \label{cor:singleclass}
\end{corollary}

\begin{proof}
  For single linkage class systems that are weakly reversible, each
  semi-locking set is a locking set.  Using Corollary
  \ref{cor:locking} completes the proof.
\end{proof}

\subsection{Connection with extreme points}
We connect our results to a condition on the extreme points of the
positive stoichiometric compatibility classes.

\begin{proposition}
  For $y \in \R^m_{\ge 0}$, let $W = \{X_i : y_i = 0\} =
  \hbox{supp}(y)^C$.  Then the following are equivalent:
  \begin{enumerate}[(i)]
  \item $y$ is an extreme point of $(y + S) \cap \R^m_{\ge 0}$.
    \label{item1}
  \item $\left[(y+S) \bigcap \R^m_{\ge0}\right] \bigcap L_W = \{y\}$.
    \label{item2}
  \end{enumerate}
  \label{prop:extreme}
\end{proposition}

\begin{proof}
  (\eqref{item1} $\implies$ \eqref{item2}) Suppose that \eqref{item1}
  is true, but that \eqref{item2} is not.  Then, because $\left[(y+S)
    \bigcap \R^m_{\ge 0}\right] \bigcap L_W$ is not discrete, there
  exists a $v \in S \bigcap L_W$ such that for sufficiently small
  $\epsilon$, $y \pm \epsilon v \in \left[(y+S) \bigcap \R^m_{\ge
      0}\right] \bigcap L_W$.  Noting that $y = (1/2)(y + \epsilon v)
  + (1/2)(y - \epsilon v)$ then shows $y$ is not an extreme point of
  $(y + S) \cap \R^m_{\ge 0}$, which is a contradiction.

  (\eqref{item2} $\implies$ \eqref{item1}) Suppose that \eqref{item2}
  is true, but that \eqref{item1} is not.  Because $y$ is not an
  extreme point of $(y + S) \cap \R^m_{\ge 0}$, there exists nonzero
  vectors $v_1 \ne y$, $v_2 \ne y$ in $(y + S) \bigcap \R^m_{\ge 0}$
  and $0 < \lambda < 1$ such that
  \begin{equation}
    y = \lambda v_1 + (1 - \lambda)v_2.
    \label{eq:lambda}
  \end{equation}
  Because $v_1,v_2 \in (y + S)\bigcap \R^m_{\ge 0}$, there exists $u,w
  \in S$ such that $v_1 = y + u$ and $v_2 = y + w$, and $u_i,w_i \ge
  0$ if $X_i \in W$.  However, because $\lambda, 1-\lambda >0$, we see
  by equation \eqref{eq:lambda} that $u_i,w_i =0$ for all $X_i \in W$.
  Therefore, $v_1,v_2 \in \left[(y+S) \bigcap \R^m_{\ge0}\right]
  \bigcap L_W$, contradicting \eqref{item2}.
\end{proof}

Theorem \ref{thm:main} can now be reformulated in the following way.

\begin{theorem}
  Let $\{\S,\C,\Re\}$ be a weakly reversible, deficiency zero, mass
  action chemical reaction system with dynamics given by equation
  \eqref{eq:main}.  For any boundary point $y \in [(c + S) \cap
  \R^m_{\ge 0}]$ of a positive stoichiometric compatibility class, let
  $W_y = \{X_i : y_i = 0\} = \hbox{supp}(y)^C$.  Finally, suppose that
  $W_y$ is a semi-locking set only if $y$ is an extreme point.  Then
  the unique positive equilibrium of each stoichiometric compatibility
  class guaranteed to exist by the Deficiency Zero Theorem is globally
  asymptotically stable relative to its compatibility class.
  \label{thm:secondary}
\end{theorem}

\begin{proof}
  Suppose $W$ is a semi-locking set.  Let $y \in L_W$ be such that
  there exists a $c \in \R^m_{>0}$ with $y \in c + S$.  If no such $y$
  and $c$ exist, then $[(c + S) \bigcap \R^m_{\ge 0}] \bigcap L_W =
  \emptyset$ for all $c \in \R^m_{>0}$.  If such a $y$ and $c$ do
  exist, then $W = W_y$ and, by assumption, $y$ is an extreme point.
  Thus, by Proposition \ref{prop:extreme}, $\left[(y+S) \bigcap
    \R^m_{\ge0}\right] \bigcap L_W = \left[(c+S) \bigcap
    \R^m_{\ge0}\right] \bigcap L_W$ is discrete. Invoking Theorem
  \ref{thm:main} completes the proof.
\end{proof}
\section{Examples}
\label{sec:examples}

We begin with an example found in \cite{ChavezThesis} for a
receptor-ligand model.  See \cite{ChavezThesis} for full details.
\begin{example}
  Consider the following system, which we assume has mass action
  kinetics
  \begin{equation}
    \begin{array}{ccc}
      2A + C & \rightleftarrows & A + D\\ 
      \left \uparrow \phantom{|} \right \downarrow & &
      \left \uparrow \phantom{|} \right \downarrow\\ 
      B + C & \leftrightarrows & E
    \end{array}.
    \label{ex:ex1}
  \end{equation}
  For this example there are four complexes, one linkage class, and
  the dimension of the stoichiometric compatibility class is easily
  verified to be three.  Therefore, the system has a deficiency of
  zero and our results apply.  The minimal semi-locking sets (that is,
  those that must be contained in all others) are given by $W_1 =
  \{A,B,E\}$, $W_2 = \{A,C,E\}$, and $W_3 = \{C,D,E\}$.  Therefore,
  showing that the set $\bigcup_{i = 1}^3 [(c + S) \cap \R^5_{\ge 0}]
  \cap L_{W_i}$ is discrete for any $c \in \R^5_{>0}$ would also show
  that the sum over all semi-locking sets is discrete.  For this
  example, it is easily verified that
  \begin{equation}
    S = \hbox{Span}\left\{ 
      \left[ \begin{array}{c}
          1\\ 0\\ 0\\ 1\\ -1
        \end{array}\right],
      \left[ \begin{array}{c}
          0\\ 1\\ 0\\ 2\\ -2
        \end{array}\right],
      \left[ \begin{array}{c}
          0\\ 0\\ 1\\ -2\\ 1
        \end{array}\right] \right\}.
    \label{eq:spanning1}
  \end{equation}
  One method to show that for a given $i$ the set $[(c + S) \cap
  \R^5_{\ge 0}] \cap L_{W_i}$ is at most discrete is to demonstrate
  that there are no non-zero vectors contained in $S$ with support
  given by $W_i^C$. This method bypasses the need to check whether the
  intersection $[(c + S) \cap \R^5_{\ge 0}] \cap L_{W_i}$ is
  non-empty.  It is easily verified that there are no non-zero vectors
  contained in $S$ with support given by $W_1^C$ or $W_2^C$.  In
  \cite{ChavezThesis} it is shown that for each $c \in \R^5_{>0}$, $(c
  + S) \cap \R^5_{\ge 0}$ does intersect one of $L_{W_1}$ or
  $L_{W_2}$.  Therefore, there are always equilibria on the boundary,
  however, by our results or those found in \cite{ChavezThesis}, they
  will not affect the global stability of the interior equilibria.

  Let $U_3 = \{x \in \R^5 \ | \ \hbox{supp}(x) \in W_3^C\}$.  It is
  easy to show that $U_3 \bigcap S = \hbox{span}\{[2, -1, 0, 0,
  0]^T\}$.  Thus, the method used in the previous paragraph does not
  work.  Therefore, for our results to apply, we need to verify that
  $[(c + S) \cap \R^5_{\ge 0}] \cap L_{W_3} = \emptyset,$ for any $c
  \in \R^5_{>0}$.  Because $L_{W_3}$ is characterized by having the
  last three entries equal to zero, in order to prove that $[(c + S)
  \cap \R^5_{\ge 0}] \cap L_{W_3} = \emptyset$ for any $c \in
  \R^5_{>0}$, it is sufficient to show that the space spanned by the
  last three entries of the vectors in \eqref{eq:spanning1} does not
  contain a vector with strictly negative components.  We have
  \begin{equation}
    \hbox{Span}\left\{ 
      \left[ \begin{array}{c}
          0\\ 1\\ -1
        \end{array}\right],
      \left[ \begin{array}{c}
          0\\ 2\\ -2
        \end{array}\right],
      \left[ \begin{array}{c}
          1\\ -2\\ 1
        \end{array}\right] \right\}
    = \hbox{Span}\left\{ 
      \left[ \begin{array}{c}
          0\\ -1\\ 1
        \end{array}\right],
      \left[ \begin{array}{c}
          1\\ -1\\ 0
        \end{array}\right] \right\},
    \label{eq:spanning2}
  \end{equation}
  which does not include a strictly negative vector.  Thus, $[(c + S)
  \cap \R^5_{\ge 0}] \cap L_{W_3} = \emptyset$ for any $c \in
  \R^5_{>0}$.  Combining all of the above with Theorem \ref{thm:main}
  shows that for any choice of rate constants and initial condition,
  the system \eqref{ex:ex1} has a globally asymptotically stable
  equilibrium value.
\end{example}

\begin{example}
  Consider the system 
  \begin{equation}
    \begin{array}{ccccc}
      2A &\rightleftarrows & A+B & \rightleftarrows & B+C \ . 
    \end{array}
    \label{ex:ex2}
  \end{equation}
  There are three complexes, one linkage class, and the dimension of
  the stoichiometric compatibility class is two.  Therefore, the
  system \eqref{ex:ex2} has a deficiency of zero.  The minimal
  semi-locking sets are $W_1 = \{A,B\}$ and $W_2 = \{A,C\}$.  The
  stoichiometric subspace is of dimension two and the quantity $A + B
  + C$ is conserved.  Thus, each stoichiometric compatibility class is
  a plane that intersects each of $L_{W_1} = \{v : v_1 = v_2 = 0, v_3
  \ne 0\}$ and $L_{W_2} = \{v : v_1 = v_3 = 0, v_2 \ne 0\}$ in
  precisely one point.  See Figure \ref{fig:example2}.  Therefore, by
  Theorem \ref{thm:main}, for any choice of rate constants and initial
  condition, the system \eqref{ex:ex2} has a globally asymptotically
  stable equilibrium value.  We note that it is easily verified that
  the eigenvalues of the linearized problem around the equilibria
  associated with the semi-locking set $W_1$ are all zero, and so the
  results of \cite{ChavezThesis} do not apply here.

  \begin{figure}
    \begin{center}
      \includegraphics[width=2in]{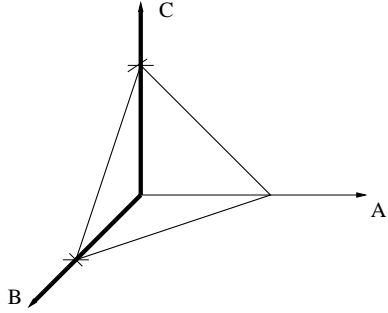}
    \end{center}
    \caption{The stoichiometric compatibility class for the system
      \eqref{ex:ex2} is a plane defined by the conservation law $A + B
      + C = M$ for some $M > 0$.  This plane intersects the sets
      $L_{W_1}$ and $L_{W_2}$ (bolded axes) each in precisely one
      point.}
    \label{fig:example2}
  \end{figure}

  \label{example2}
\end{example}

\begin{example}
  Consider the system
  \begin{equation}
    \begin{array}{ccc}
       2A & \rightleftarrows & A+B\\
    \end{array} \hspace{.5in}
    \begin{array}{ccc}
      2B & \rightleftarrows & A+C \ .\\
    \end{array}
    \label{ex:ex3}
  \end{equation}
  There are four complexes, two linkage classes, and the dimension of
  the stoichiometric compatibility class is two.  Therefore, the
  system \eqref{ex:ex3} has a deficiency of zero.  The only minimal
  semi-locking set is $W = \{A,B\}$, and this is also a locking set.
  The stoichiometric subspace is of dimension two and the quantity $A
  + B + C$ is conserved. Thus, each stoichiometric compatibility class
  is a plane that intersects $L_W = \{v : v_1 = v_2 = 0, v_3 \ne 0\}$
  in precisely one point.  Therefore by Theorem \ref{thm:main} or
  Corollary \ref{cor:locking}, for any choice of rate constants and
  initial condition, the system \eqref{ex:ex3} has a globally
  asymptotically stable equilibrium value. It is easily verified that
  the boundary equilibria are not hyperbolic with respect to their
  compatibility classes, and so the results of \cite{ChavezThesis} do
  not apply in this case.
\end{example}

\section{Non-mass action kinetics}
\label{sec:more}

In \cite{Sontag2001}, Sontag extended the Deficiency Zero Theorem to
systems with non-mass action kinetics.  He considered weakly
reversible, deficiency zero system whose kinetic functions are given
by
\begin{equation}
  R_{y \to y'}(x) = k_{y \to y'}\theta(x_1)^{y_1} \cdots \theta(x_m)^{y_m},
  \label{eq:nonmass}
\end{equation}
where the functions $\theta_i:\R \to [0,\infty)$ satisfy the following:
\begin{enumerate}
\item Each $\theta_i$ is locally Lipschitz.
\item $\theta_i(0) = 0$.
\item $\int_0^1 |\ln(\theta_i(y)|dy < \infty$.
\item The restriction of $\theta_i$ to $\R_{\ge 0}$ is strictly
  increasing and onto $\R_{\ge 0}$.
\end{enumerate}
To prove the local asymptotic stability of the unique equilibrium
within each stoichiometric compatibility class the following Lyapunov
function was used
\begin{equation}
  V(x) = \sum_{i = 1}^m \int_{\bar x_i}^{x_i}\left( \rho_i(s) -
    \rho_i(\bar x_i) \right) ds, 
  \label{eq:lyapunov2}
\end{equation}
where $\rho_i(s) = \ln \theta_i(s)$ and $\bar x$ is the unique
equilibrium within the positive stoichiometric compatibility class.
Note that $\theta(x) = |x|$ gives mass action kinetics, in which case
the Lyapunov function given in equation \eqref{eq:lyapunov2} is the
same as that in equation \eqref{eq:lyapunov}.  The only dynamical
property of the Deficiency Zero Theorem used in this paper was that
$\nabla V(x) \to -\infty$ as $x_i \to 0$.  We note that for the
Lyapunov function \eqref{eq:lyapunov2}, we still have that property
because
\begin{equation*}
\nabla V(x) = \sum_{i=1}^m \rho_i(x_i) - \rho_i(\bar x_i),
\end{equation*}
and $\rho_i(x_i) = \ln \theta_i(x_i) \to -\infty$ as $x_i \to 0$ by
the properties of $\theta_i(\cdot)$ given above.  Therefore, our
results in this paper, and in particular Theorem \ref{thm:main},
Corollary \ref{cor:locking}, Corollary \ref{cor:singleclass}, and
Theorem \ref{thm:secondary}, are valid in the setting
\eqref{eq:nonmass}.

\vspace{.25in}
\begin{center}
  {\bf \large Acknowledgments}
\end{center}

I would like to thank Gheorghe Craciun for several illuminating
conversations pertaining to chemical reaction systems.  I would also
like to thank Martin Feinberg for pointing out the connection with
extreme points.  This work was done under the support of the grants
NSF-DMS-0109872 and NSF-DMS-0553687.

\bibliographystyle{amsplain}
\bibliography{AndGlobalRev}

\end{document}